\documentclass[10pt]{article}

\usepackage{amsmath,amssymb,amsthm,amsfonts,amsbsy, enumerate}
\usepackage[dvips]{epsfig}
\usepackage{comment}

\usepackage{tikz}
\usetikzlibrary{shapes}

\newfont{\bb}{msbm10}
\def\Bbb#1{\mbox{\bb #1}}

\def\di{\mathop{\diamond}}

\def\R{ {\Bbb R} }

\def\Nul{\mathop{\rm Nul}}

\def\rank{\mathop{\rm rank}}

\def\n2{\frac{n}{2}}
\def\L{\lambda}

\def\vecone{{\it 1}}
\def\bp{{\rm bp}}

\def\t{^\top}

\def\vform{\it }

\def\vx{{\vform x}}
\def\vy{{\vform y}}

\def\v0{{\vform 0}}

  \newcommand{\D}{\mathcal{D}}

\newtheorem{theorem}{Theorem}[section]
\newtheorem{proposition}[theorem]{Proposition}
\newtheorem{lemma}[theorem]{Lemma}
\newtheorem{corollary}[theorem]{Corollary}
\newtheorem{example}[theorem]{Example}

\newtheorem{remark}[theorem]{Remark}

\newtheorem{question}[theorem]{Question}

\begin{document}

\tikzstyle{place}=[draw,circle,minimum size=0.5mm,inner sep=1pt,outer sep=-1.1pt,fill=black]

\title{Addressing Graph Products and Distance-Regular Graphs}
\author{Sebastian M. Cioab\u{a}\footnote{Department of Mathematical Sciences, University of Delaware, Newark, DE 19716-2553, USA; {\tt cioaba@udel.edu}. Research partially supported by NSF grant DMS-1600768.}\, , Randall J. Elzinga\footnote{Department of Mathematics, Royal Military College, Kingston, ON K7K 7B4, Canada; {\tt rjelzinga@gmail.com}. Current address: Info-Tech Research Group, London, ON, N6B 1Y8, Canada.}\, , \\Michelle Markiewitz\footnote{Department of Mathematical Sciences, University of Delaware, Newark, DE 19716-2553, USA; {\tt mmark@udel.edu}. Research supported by the Summer Scholars Undergraduate Program at the University of Delaware.}\, , Kevin Vander Meulen\footnote{Department of Mathematics, Redeemer University College, Ancaster, ON L9K 1J4, Canada; {\tt kvanderm@redeemer.ca}. Research supported
by NSERC Discovery Grant 203336.}\, , and
Trevor Vanderwoerd\footnote{Department of Mathematics, Redeemer University College, Ancaster, ON L9K 1J4, Canada; {\tt tvanderwoerd@redeemer.ca.}}}
\date{\today}
\maketitle


\begin{abstract}
Graham and Pollak showed that the vertices of any connected graph $G$ can be assigned $t$-tuples with entries in $\{0, a, b\}$, called addresses, such that the distance in $G$ between any two vertices equals the number of positions in their addresses where one of the addresses equals $a$ and the other equals $b$. In this paper, we are interested in determining the minimum value of such $t$ 
for various families of graphs. We develop two ways to obtain this value for the
Hamming graphs and present a lower bound for the triangular graphs. 


\end{abstract}

\section{Graph Addressings}

A $t$-{\em address} is a $t$-tuple with entries in $\{0,a,b\}$. 
An {\em addressing} of length $t$ for a graph $G$ is an assignment of $t$-addresses to the vertices of $G$ so that the distance between two vertices is equal to the number of locations in the addresses at which one of the addresses equals $a$ and the other address equals $b$. For example, we have a $3$-addressing of a graph in Figure~\ref{sample}. Graham and Pollak \cite{GP1} introduced such addressings, using symbols $\{*, 0,1\}$ instead $\{0,a,b\}$, in the context of loop switching networks.

\begin{figure}[ht]
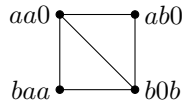

\begin{center}$
\tikzpicture 
\node (1) at (-0.5,0.5)[place] {};
\node (2) at (0.5,0.5)[place] {};
\node (3) at (0.5,-0.5)[place] {};
\node (4) at (-0.5,-0.5)[place] {};
\draw [right] (3) to (4);
\draw [right] (1) to (2);
\draw [right] (2) to (3);
\draw [right] (3) to (1);
\draw [right] (1) to (4);
\node [right] at (2) {$ab0$};
\node [right] at (3) {$b0b$};
\node [left] at (4) {$baa$};
\node [left] at (1) {$aa0$};
\endtikzpicture$\end{center}
\caption{A graph addressing}\label{sample}
\end{figure}

We are interested in the minimum $t$ such that $G$ has an addressing of length $t$. We denote
such a mininum by $N(G)$. Graham and Pollak \cite{GP1,GP2} showed that $N(G)$ equals the
biclique partition number of the distance multigraph of $G$. Specifically, 
the \emph{distance multigraph of $G$}, $\D (G)$, is the multigraph with the same vertex set as $G$ where the multiplicity of any edge $uv$ equals the distance in $G$ between vertices $u$ and $v$. The \emph{biclique partition number} $\bp(H)$ of a multigraph $H$ is the minimum number of complete bipartite subgraphs (bicliques) of $H$ whose edges partition the edge-set of $H$. This parameter and its close covering variations have been studied by several researchers and appear in different contexts such as computational complexity or geometry (see  
 for example, \cite{CT,GP1,GP2,GSW,H,KRW,RSV,Zaks}). Graham and Pollak deduced that $N(G)\leq r(n-1)$ for any connected $G$ of order $n$ and diameter $r$ and conjectured that $N(G)\leq n-1$ for any connected graph $G$ of order $n$. This conjecture, also known as the {\em squashed cube conjecture}, was proved by Winkler \cite{Winkler}.

To bound $N(G)$ below, Graham and Pollak used an eigenvalue argument on the adjacency matrix of
$\mathcal{D}(G)$. Specifically, if $M$ is a symmetric real matrix, 
let  $n_{+}(M), n_{-}(M),$ and  $n_0(M)$ denote the number of eigenvalues of $M$ (including multiplicity) that are positive, negative and zero, respectively. The \emph{inertia} of $M$ is the triple $(n_{+}(M), n_{0}(M), n_{-}(M))$.
The adjacency matrix of $\mathcal{D}(G)$ 
will be denoted by $D(G)$; we will also refer to $D(G)$ as the distance matrix of $G$.
The inertia of distance matrices has been studied by various authors for many classes of graphs \cite{AP,I,KS,ZG}.
Witsenhausen (cf. \cite{GP1,GP2}) showed that 
\begin{equation}\label{Witsenhausen_ineq}
N(G)\geq \max(n_+(D(G)),n_{-}(D(G))).
\end{equation}
Letting $J_n$ denotes the all one $n\times n$ matrix and $I_n$ denotes the $n\times n$ identity matrix, 
and noting that $n_{-}(D(K_n))= n_{-}(J_n-I_n)$,  
Graham and Pollak \cite{GP1,GP2} used the bound (\ref{Witsenhausen_ineq}) 
to conclude that 
\begin{equation}\label{GPn}
N(K_n)=n-1.
\end{equation}

Graham and Pollak \cite{GP1,GP2} also determined $N(K_{n,m})$ for many values of $n$ and $m$. The determination of $N(K_{n,m})$ for all values of $n$ and $m$ was completed by Fujii and Sawa~\cite{FS}. A more general addressing scheme, allowing the addresses to contain more than two different nonzero symbols, was recently studied by Watanabe, Ishii and Sawa~\cite{WIS}. The parameter $N(G)$ has been determined when $G$ is a tree or a cycle \cite{GP2}, as well as one particular triangular graph $T_4$ \cite{Zaks}, described in Section~\ref{triangleS}. For the Petersen graph $P$, Elzinga, Gregory and Vander Meulen \cite{EGV} showed that $N(P)=6$. To the best of our knowledge, these are the only graphs $G$ for which addressings 
of length $N(G)$ have been determined. We will say a $t$-addressing of $G$ is \emph{optimal}
if $t=N(G)$. An addressing is \emph{eigensharp} \cite{KRW} if equality is obtained in (\ref{Witsenhausen_ineq}).

In this paper, we study optimal addressings of Cartesian graph products and the distance-regular graphs known as
triangular graphs. Let $H(n,q)$ is the Hamming graph 
whose vertices are the $n$-tuples over an alphabet with $q$ letters with two $n$-tuples being adjacent if and only if their Hamming distance is $1$. We give two different proofs showing that $N(H(n,q))=n(q-1)$. 
This generalizes the Graham-Pollak result (\ref{GPn}) since $H(1,q)=K_q$. We determine
that the triangular graphs are not eigensharp.

\section{Addressing Cartesian Products}\label{CP}
Suppose that $G_i, i=1,\ldots,k$ are graphs and that each graph $G_i$ has vertex set $V(G_i)$
and order $n_i=|V(G_i)|$.
The {\em Cartesian product} $G_1\Box G_2 \Box \cdots \Box G_k$ of  $G_1, G_2, \ldots, G_k$
is the graph
with vertex set $V(G_1)\times V(G_2) \times \cdots \times V(G_k)$,
order $n=n_1 n_2 \cdots n_k$, and with
vertices $\vx=(x_1,\ldots,x_k)$ and $\vy=(y_1,\ldots,y_k)$
adjacent if for some index $j$, $x_j$ is adjacent to $y_j$ in $G_j$
while $x_i=y_i$ for all remaining indices $i\ne j$.
Thus, if $d$ and $d_i$ denote distances in $G$ and $G_i$, respectively,
then
\begin{eqnarray}\label{eq:addpropdiam}
d(\vx,\vy) = \sum_{i=1}^k d_i(x_i,y_i)
\end{eqnarray}

It follows that if each $G_i$, $i=1,\ldots,k$ is given an addressing,
then each vertex $\vx$ of $G$ may be addressed
by concatenating the addresses of its components $x_i$.
Therefore, the parameter $N$ is subadditive on Cartesian products;  that is, if
\begin{eqnarray}\label{eq:cartprodG}
 G=G_1 \Box \cdots \Box G_k
\end{eqnarray}
then
\begin{eqnarray} \label{in:subaddN}
N(G) \le N(G_1) + \cdots + N(G_k)
\end{eqnarray}
Note that $N(G_1) + \cdots + N(G_k)\leq \left( \sum_{i=1}^k n_i\right) -k \leq \left( \prod_{i=1}^kn_i \right) -1 =n-1$. Thus (\ref{in:subaddN})
can improve on Winkler's upper bound of $n-1$ when $G$ is a Cartesian product. 
\begin{question}\label{Qequality}
 Must equality holds in (\ref{in:subaddN})
for all choices of $G_i$? {\rm Remark \ref{rem:countereg} might provide a possible counterexample.}
\end{question}

\section{Distance Matrices of Cartesian Products}\label{CPD}

If $v_1, \ldots, v_n$ denote the vertices of a connected graph $G$, the distance matrix $D(G)$ of $G$ is the $n\times n$ matrix with entries $D(G)_{ij}=d(v_i,v_j)$.  Because $G$ is connected,
its adjacency matrix $A(G)$ and its distance matrix $D(G)$ are irreducible symmetric nonnegative integer matrices and by the Perron-Frobenius Theorem (see \cite[Proposition 3.1.1]{BHBook} or \cite[Theorem 8.8.1]{GR}), the largest eigenvalue 
of each of these matrices has multiplicity 1. We call this largest eigenvalue
the \emph{Perron} value of the matrix and often denote it by $\rho$.

To obtain a formula for the distance matrix of a Cartesian product of graphs, we will use an additive analogue of the Kronecker product of matrices. Recall first that if $A$ is an $n \times n$ matrix
and  $B$ an $m\times m$ matrix, with $x\in\R^n$,  $y\in \R^m$, then
the Kronecker products $A\otimes B$ and $\vx\otimes\vy$  are defined by
\begin{equation}\label{usualKronecker}
A\otimes B=\left[\begin{array}{cccc}
a_{11}B&a_{12}B&\cdots&a_{1m}B\\
a_{21}B&a_{22}B&\cdots&a_{2m}B\\
\vdots&\vdots& &\vdots\\
a_{m1}B&a_{m2}B&\cdots&a_{mm}B\\
\end{array}\right] \qquad
\rm{ and}\qquad
\vx\otimes\vy=\left[
\begin{array}{c}
x_1y\\
x_2y\\
\vdots\\
x_ny\\
\end{array}
\right]
\end{equation}
For the additive analogue,
we use the symbol $\di$ and define $A\di B$ and $\vx \di \vy$ by
\begin{equation}
A\di B=\left[\begin{array}{cccc}
a_{11}+B&a_{12}+B&\cdots&a_{1m}+B\\
a_{21}+B&a_{22}+B&\cdots&a_{2m}+B\\
\vdots&\vdots& &\vdots\\
a_{m1}+B&a_{m2}+B&\cdots&a_{mm}+B\\
\end{array}\right] \qquad
\rm{ and}\qquad
\vx\di\vy=\left[
\begin{array}{c}
x_1+y\\
x_2+y\\
\vdots\\
x_n+y\\
\end{array}
\right]
\end{equation}
If  $G=G_1\Box G_2 \Box \cdots \Box G_k$,
then the additive property (\ref{eq:addpropdiam}) implies that\footnote{This approach 
was suggested by the late David A. Gregory.}
\begin{eqnarray} \label{eq:diamondprodD}
D(G)=D(G_1)\di D(G_2) \di \cdots \di D(G_k)
\end{eqnarray}
Observe that 
\begin{equation}\label{eqn:diamond}
A\di B = A\otimes J_m + J_n \otimes B \ \mbox{ and } \ \vx\di \vy = \vx\otimes \vecone_m + \vecone_n \otimes \vy
\end{equation}
where $\vecone_n\in \R^n$ denotes the column vector whose entries are all one.
Let $\v0_n\in\R^n$ denote the column vector with all zero entries.  The following two lemmas are
 due to D.A. Gregory.

\begin{lemma} \label{lem:eigvec}
Let $A$ and $B$ be $n\times n$ and $m\times m$ real matrices respectively.
If $A \vx=\L\vx$  and $\vx\t  \vecone_n = \sum x_i = 0$,
then $(A\di B)(\vx\di \v0_m) = m\L (\vx\di\v0_m)$. 
Also, if $\vecone_m^Ty=0$ and $By=\mu y$,
then $(A\di B)(\v0_n\di \vy) = n\mu (\v0_n\di\vy)$.
\end{lemma}

\begin{proof} We use properties of Kronecker products: 
\begin{eqnarray*}
(A\di B )(x\di \v0) &=& (A\otimes J_m + J_n\otimes B)(x\otimes \vecone_m+ \vecone_n\otimes \v0)\\
 &=& Ax\otimes J_m\vecone_m + J_nx\otimes B\vecone_m+ A\vecone_n\otimes J_m0 +J_n\otimes B\v0\\
 &=& Ax\otimes J_m\vecone_m = \lambda m (x\otimes \vecone_m)\\
 &=& \lambda m (x\otimes \vecone_m + \vecone_n \otimes \v0) = \lambda m (x \di \v0)\vspace{-1em}
\end{eqnarray*} 
A similar argument work the vector $(\v0_n\di \vy)$.
\end{proof}

Throughout we will say a square matrix is $k$-\emph{regular} if it has constant row sum $k$.
\begin{lemma}\label{lem:Perron}
If $A$ is $\rho_A$-regular and
$B$ is $\rho_B$-regular
then $A\di B$ is ($m\rho_A+n\rho_B$)-regular.
\end{lemma}

\begin{proof} Using properties of Kronecker products,
\begin{eqnarray*}
(A\di B )(\vecone_n\otimes \vecone_m) &=& (A\otimes J_m + J_n\otimes B)(\vecone_n\otimes \vecone_m)\\
 &=& A\vecone_n \otimes J_m\vecone_m + J_n\vecone_n \otimes B\vecone_m)\\
 &=& \rho_Am (\vecone_n\otimes \vecone_m) + n\rho_B(\vecone_n\otimes \vecone_m)\\
 &=& (\rho_Am+n\rho_B)(\vecone_n\otimes \vecone_m).\vspace{-1em}
\end{eqnarray*}
Thus $(A\di B)\vecone = (\rho_Am+n\rho_B) \vecone$.
\end{proof}

\begin{lemma} \label{cor:supaddinertia}
If  $G=G_1\Box G_2 \Box \cdots \Box G_k$ and each $\D(G_i), i=1,\ldots,k$
is regular then
\begin{enumerate}
  \item[(a)] $n_-(D(G)) \ge \sum_i n_-(D(G_i))$, and
  \item[(b)] $n_+(D(G)) \ge 1+\sum_i (n_+(D(G_i)) -1)$.
\end{enumerate}
\end{lemma}

\begin{proof}
Because $\D(G_i)$  regular,
$D(G_i)\vecone_{n_i}=\rho_i \vecone_{n_i}$ where $\rho_i$ is the Perron value of $D(G_i)$.
Then $\vecone_{n_i}$
is a $\rho_i$-eigenvector of $D(G_i)$ and $\R^{n_i}$ has an orthogonal
basis of eigenvectors of $D(G_i)$ that includes $\vecone_{n_i}$ as a member.
Thus,  Lemma~\ref{lem:eigvec}
with $A=D(G_i)$ and $B=D(\Box_{j\ne i}G_j)$ implies that
the $n_i-1$ eigenvectors of $D(G_i)$ in the basis other than $\vecone_{n_i}$
contribute $n_i-1$ orthogonal eigenvectors to the matrix $D(G)$.

An eigenvector of $D(G_i)$
with eigenvalue $\L\ne \rho_i$
contributes an eigenvector of $D(G)$ with eigenvalue $\L(n_1n_2 \cdots n_k)/n_i=\L n/n_i$.
This eigenvalue has the same sign as $\L$ if $\L\ne 0$. Also, if $i\ne j$, 
then each of the $n_i-1$ eigenvectors contributed to $D(G)$ by $D(G_i)$ is orthogonal to 
each of the analogous $n_j-1$ eigenvectors contributed to $D(G)$ by $D(G_j)$.
Thus, the inequality (a) claimed for $n_-$ follows. Also,
by Lemma \ref{lem:eigvec}, $\vecone_n$ is an
eigenvector of $D(G)$ with a positive eigenvalue $\rho$,  so the
inequality (b) for $n_+$ follows.
\end{proof}

\begin{remark} \label{rem:countereg} {\rm (Observed by D.A. Gregory) The inequality in Lemma \ref{cor:supaddinertia}
need not hold if the regularity assumption is dropped.
For example, suppose $G=G_1\Box G_1$ where $G_1$ is the graph on $6$ vertices obtained $K_{2,4}$
by inserted an edge incident to the two vertices in the part of size $2$. 
Then $n_-(D(G_1))=5$ but $n_-(D(G))=9<5+5$.
Also, $N(G_1)=5$, so  $9\le N(G)\le 10$ by (\ref{Witsenhausen_ineq}) and (\ref{in:subaddN}). 
An affirmative answer to Question~\ref{Qequality} would imply $N(G)=10$.
}\end{remark}

If each $D(G_i)$ in (\ref{eq:diamondprodD}) is regular, then Lemma \ref{cor:supaddinertia}
accounts for $1+\sum_i(\rank D(G_i)-1)= k-1+ \sum_i \rank D(G_i)$ of the $\rank D(G)$
nonzero eigenvalues of $D(G)$.
The following results imply that if each $D(G_i)$ is regular then all of the remaining eigenvalues must be equal to zero.  Equivalently, the results will imply that if each $D(G_i)$ in (\ref{eq:diamondprodD}) is regular, then
equality must hold in Lemma \ref{cor:supaddinertia}(a) and (b).

The next result (proved by D.A. Gregory) is obtained by exhibiting an orthogonal basis of $\R^{mn}$ consisting of eigenvectors of $A\di B$ when $A$ and $B$
are symmetric and regular.

\begin{theorem}\label{lem:eigvecsAdiB} 
 Let $A$ be a regular symmetric real $n\times n$ matrix with    $A\vecone_n=\rho_A \vecone_n$ with $\rho_A>0$
 and let $B$ be a regular symmetric matrix of order $m$ with $B\vecone_m = \rho_B \vecone_m$ with $\rho_B>0$. Then
 \begin{enumerate}
 \item[(a)] $n_-(A\di B) =  n_-(A)+n_-(B)$,
 \item[(b)] $n_+(A\di B) = n_+(A) + n_+(B) -1$, and
 \item[(c)] $n_o(A\di B) = nm -n -m + 1 + n_o(A) + n_o(B)$.
 \end{enumerate}
\end{theorem}

\begin{proof} As in Lemma \ref{cor:supaddinertia}, Lemma \ref{lem:eigvec}
can be used to provide eigenvectors that imply that $n_-(A\di B) \ge  n_-(A)+n_-(B)$ and
$n_+(A\di B) \ge n_+(A) + n_+(B) -1$. It remains to
exhibit an adequate number of linearly independent eigenvectors of $A\di B$ for the eigenvalue $0$.

If $\vecone_n\t x = 0$ and $\vecone_m\t y=0$, then
$$(A\di B )(x\otimes y) = (A\otimes J_m + J_n\otimes B)(x\otimes y)= Ax\otimes 0_m + 0_n\otimes By =0_{nm}$$
This accounts for at least $(n-1)(m-1)=nm-n-m+1$ orthogonal eigenvectors  of $A\di B$ with eigenvalue $0$.
Moreover, if $Au=0$ then $\vecone_n\t u =0$ and hence, by Lemma \ref{lem:eigvec}, $(A\di B)(u\di 0_m)=0_{nm}$.
Likewise, if $Bv=0$ then $\vecone_m\t v =0$ and by Lemma~\ref{lem:eigvec}, $(A\di B)(0_n\di v)=0_{nm}$. If each set of vectors $x$, each set of vectors $y$, each set of vectors $u$ and each set of vectors $v$ that occur above are chosen to be orthogonal, then the resulting vectors
$x\otimes y, u\otimes \vecone_m, \vecone_n\otimes v$ will be orthogonal.
Thus, $n_o(A\di B) \ge nm -n -m + 1 + n_o(A) + n_o(B)$.  Adding the three inequalities obtained above, we get
\begin{eqnarray*}
nm & = & n_-(A\di B) + n_+(A\di B) + n_o(A\di B) \\
   & \ge & n_-(A)+n_-(B) + n_+(A) + n_+(B) -1 + nm -n -m + 1 + n_o(A) + n_o(B)\\
   & = & nm.
\end{eqnarray*}
Thus equality holds in each of the three inequalities.
\end{proof}

\begin{corollary} \label{cor:addinertia} 
If  $G=G_1\Box G_2 \Box \cdots \Box G_k$ and each $D(G_i), i=1,\ldots,k$
is regular then
\begin{enumerate}
  \item[(a)] $n_-(D(G)) = \sum_i n_-(D(G_i))$, and
  \item[(b)] $n_+(D(G)) = 1+\sum_i (n_+(D(G_i)) -1)$.
\end{enumerate}
\end{corollary}

\begin{remark}\label{rem:nullbound}{\rm In the proof of Theorem \ref{lem:eigvecsAdiB}, whether or not $A$ and $B$
are symmetric and regular, we always have
$(A\di B )(x\otimes y)=0_{nm}$ whenever $\vecone_n\t x = 0$ and $\vecone_m\t y=0$.
Thus,
$$\Nul(A\di B) \ge (n-1)(m-1)$$
for all square matrices $A$ and $B$ of orders
$n$ and $m$, respectively.
}\end{remark}

In order to apply Lemma \ref{cor:supaddinertia} to the Cartesian product
(\ref{eq:cartprodG}), it would be helpful to have conditions on the graphs  $G_i$ that imply that the distance matrices $D(G_i)$ are regular.  The following remark gives a few examples
of graphs whose distance matrix has constant row sums.

\begin{remark} (Regular distance matrices)

{\rm
1. If $G$ is either distance regular or vertex transitive, then $D(G)$ is $\rho$-regular
where $\rho$ is equal to the sum of all the distances from a particular vertex
to each of the others.

2. If $G$ is a regular graph of order $n$ and the diameter of $G$ is either one or two,
then $D(G)$ is $\rho$-regular with $\rho=2(n-1)-\rho_A$ where $\rho_A$ is the
Perron value of the adjacency matrix $A$ of $G$.
For if $A$ is the adjacency matrix of $G$, then $D(G)= A+2(J_n-I_n-A)= 2(J_n-I_n)-A$.
This holds, for example, when  $G$ is the Petersen graph or $G=K_n$ (the complete graph on $n$ vertices) or when $G=K_{m,m}$ (the complete balanced bipartite graph on $n=2m$ vertices).
}\end{remark}

\begin{question}\label{Qconditions}{What are other conditions on a graph that imply that its distance matrix
 is regular?
} \end{question}

\begin{theorem} \label{lem:Kev} Let $G=G_1\Box G_2 \Box \cdots \Box G_k$.
If $D(G_i)$ is regular  and $N(G_i)=n_-(D(G_i))$ for $i=1,\ldots,k,$
then $N(G)=\sum_{i=1}^k N(G_i)$.
\end{theorem}

\begin{proof}
By the lower bound \eqref{Witsenhausen_ineq} and the subadditivity property \eqref{in:subaddN},  $\sum_i N(G_i)\ge N(G)\ge n_-D(G)$, where
by Lemma \ref{cor:supaddinertia}(a), $n_-(D(G))\ge \sum_i n_-(D(G_i))= \sum_i N(G_i)$.
\end{proof}

\begin{example}\label{EHamming}{\rm The Cartesian product of complete graphs, 
$G=K_{n_1}\Box K_{n_2} \Box \cdots \Box K_{n_k}$ is also known as
a Hamming graph.
By (\ref{GPn}) and Theorem~\ref{lem:Kev}, it follows that 
  $N(G)= \sum_{i=1}^k (n_i-1)$.
In the next section, we explore this result using a different description of the Hamming graphs.
}\end{example}

\section{Optimal Addressing of Hamming Graphs}\label{ham}

Let $n\geq 1$ and $q\geq 2$ be two integers. The vertices of the Hamming graph $H(n,q)$ can be
described as the words of length $n$ over the alphabet $\{1,\dots,q\}$. Two vertices $(x_1,\dots,x_n)$ and $(y_1,\dots,y_n)$ are adjacent ifand only if their Hamming distance is $1$. If $n=1$, $H(1,q)$ is the complete graph $K_q$. The following result, 
can be derived from Example~\ref{EHamming}, but we provide another interesting and constructive argument.
\begin{theorem}
If $n\geq 1$ and $q\geq 2$, then $N(H(n,q))=n(q-1)$.
\end{theorem}
\begin{proof}
We first prove that the length of any addressing of $H(n,q)$ is at least $n(q-1)$. For $0\leq k\leq n$, let $A_k$ denote the distance $k$ adjacency matrix of $H(n,q)$. The adjacency matrix of the distance multigraph of $H(n,q)$ is $D(H(n,q))=\sum_{k=1}^{n}kA_k$. The graph $H(n,q)$ is distance-regular and therefore, $A_1,\dots, A_n$ are simultaneously diagonalizable. The eigenvalues of the matrices $A_1,\dots,A_n$ were determined by Delsarte in his thesis \cite{Delsarte} (see also \cite[Theorem 30.1]{vLW}).
\begin{proposition}
Let $k\in \{1,\dots,n\}$. The eigenvalues of $A_k$ are given by the Krawtchouk polynomials:
\begin{equation}
\lambda_{k,x}=\sum_{i=0}^{k}(-q)^i(q-1)^{k-i}{n-i\choose k-i}{x\choose i}
\end{equation}
with multiplicity ${n\choose x}(q-1)^x$ for $x\in \{0,1,\dots,n\}$.
\end{proposition}

The Perron value of $A_k$ equals ${n\choose k}(q-1)^{k}$. Thus, the Perron value of $D(H(n,q))$ equals $\sum_{k=1}^{n}{n\choose k}k(q-1)^k=nq^{n-1}(q-1)$ and has multiplicity one. The other eigenvalues of $D(H(n,q))$ are 
\begin{align*}
\mu_x&=\sum_{k=1}^{n}k\lambda_{k,x}=\sum_{k=1}^{n}k\sum_{i=0}^{k}(-q)^i(q-1)^{k-i}{n-i\choose k-i}{x\choose i}\\
&=\sum_{i=0}^{n}(-q)^i{x\choose i}\sum_{k=i}^{n}k(q-1)^{k-i}{n-i\choose k-i}=\sum_{i=0}^{n}(-q)^i{x\choose i}\sum_{t=0}^{n-i}(i+t)(q-1)^t{n-i\choose t}\\
&=\sum_{i=0}^{n}(-q)^i{x\choose i}\left(nq^{n-i}-(n-i)q^{n-i-1}\right)\\
&=q^{n-1}\sum_{i=0}^{n}{x\choose i}(-1)^{i}i=\begin{cases}-q^{n-1} & \text{ if } x=1\\
0 & \text{ if } x\geq 2.
\end{cases}
\end{align*}
with multiplicity ${n\choose x}(q-1)^x$ for $1\leq x\leq n$. Thus, the spectrum of $D(H(n,q))$, with multiplicities, is 
\begin{equation}
\begin{pmatrix}
nq^{n-1}(q-1) &  -q^{n-1} &  0\\
1 & n(q-1) & q^n-1-q(n-1)
\end{pmatrix}.
\end{equation}
where the first row contains the distinct eigenvalues of $D(H(n,q))$ and the second row contains their multiplicities. Thus, $\max(n_{-}(D(H(n,q))),n_{+}(D(H(n,q))))=n(q-1)$ and Witsenhausen's inequality \eqref{Witsenhausen_ineq} implies that $N(H(n,q))\geq n(q-1)$.

To show $n(q-1)$ is the optimal length of an addressing of $H(n,q)$, we describe a partition of the edge-set of the distance multigraph of $H(n,q)$ into exactly $n(q-1)$ bicliques. For $1\leq i\leq n$ and $1\leq t\leq q-1$, define the biclique $B_{i,t}$ whose color classes are 
$$
\{(x_1,\dots,x_n): x_i=t \}
$$
and 
$$
\{(x_1,\dots, x_n): x_i\geq t+1\}.
$$
One can check easily that if $u$ and $v$ are two distinct vertices in $H(n,q)$, there exactly $d_H(u,v)$ bicliques $B_{i,t}$ containing the edge $uv$. Thus, the $n(q-1)$ bicliques $B_{i,t}$ partition the edge set of the distance multigraph of $H(n,q)$ and $N(H(n,q))\leq n(q-1)$. This finishes our proof.
\end{proof}

We remark here that the spectrum of the distance matrix of $H(n,q)$ was also computed by Indulal \cite{I}.

\section{Triangular Graphs
}\label{triangleS}

The triangular graph $T_n$ is the line graph of the complete graph $K_n$ on $n$ vertices. When $n\geq 4$, the triangular graph $T_n$ is a strongly regular graph with parameters $\left({n\choose 2}, 2(n-2),n-2,4\right)$. The adjacency matrix of $T_n$ has spectrum
\begin{equation}
\begin{pmatrix}
2(n-2) &  n-4 &  -2\\
1 & n-1 & {n\choose 2}-n
\end{pmatrix}
\end{equation} 
and therefore, the distance matrix $D(T_n)$ has spectrum 
\begin{equation}
\begin{pmatrix}
(n-1)(n-2) &  2-n &  0\\
1 & n-1 & {n\choose 2}-n
\end{pmatrix}
\end{equation} 
Witsenhausen's inequality \eqref{Witsenhausen_ineq} implies that $N(T_n)=bp(D(T_n))\geq n-1$ for $n\geq 4$.

The problem of addressing $T_4$ is equivalent to determining the biclique partition number of the multigraph obtained from $K_6$ by adding one perfect matching. This formulation of the problem was studied by Zaks \cite{Zaks} and Hoffman \cite{H} (see also 
Section~\ref{K222}). Zaks proved that $N(T_4)=4$ and hence $T_4$ is not eigensharp. 
We will reprove the lower bound of Zaks~\cite{Zaks} in Lemma~\ref{t4notegnsharp} using a technique of \cite{EGV}. 
The argument of Lemma~\ref{t4notegnsharp} will then be 
used to show that $T_n$ is not eigensharp for any $n\geq 4$ in Theorem~\ref{tnnotegnsharp}.

The \emph{addressing matrix} of a $t$-addressing is the $n$-by-$t$ 
matrix $M(a,b)$ where the $i$-th row of $M(a,b)$ is the address of vertex $i$. $M(a,b)$ can be written as a function of $a$ and $b$:
\[ M(a,b) = aX + bY, \] where $X$ and $Y$ are matrices with entries in $\{0,1\}$. Elzinga et al.~\cite{EGV}  use the addressing matrix, along with results from Brandenburg et al. \cite{BGK} and Gregory et al. \cite{GSW}, to create the following theorem:

\begin{theorem}\cite{EGV}\label{egnsharp}
Let $M(a,b)$ be the address matrix of an eigensharp addressing of a graph $G$. Then for all real scalars $a,b$, each column of $M(a,b)$ is orthogonal to the null space of $D(G)$. Also, the columns of $M(1,0)$ are linearly independent, as are the columns of $M(0,1)$.
\end{theorem}

 In~\cite[Theorem 3]{EGV}, Elzinga et al. use Theorem~\ref{egnsharp} to show that the Petersen graph does not have an eigensharp addressing. We will use a similar approach on triangular graphs. 

\begin{lemma}\label{t4notegnsharp}
The triangular graph $T_4$ is not eigensharp, that is, $N(T_4) \geq 4$.
\end{lemma}
\begin{proof} Suppose $T_4$ is eigensharp. Let $D=D(T_4)$. By Theorem \ref{egnsharp}, all vectors in the null space of $D$ are orthogonal to the columns of a $6 \times 3$ addressing matrix $M(a,b)$.

We can construct null vectors of $D$ in the following manner, referring to the entries of the null vector as 
\emph{labels}: choose any two non-adjacent vertices, and label them with zeroes. The remaining four vertices form a $4$-cycle, which will be alternatingly labelled with $1$ and $-1$, as in Figure~\ref{FigT4}.
\begin{figure}[ht]
\[
\begin{tikzpicture}[line join=bevel,z=-5.5]
\coordinate (A1) at (0,0,-1.25);  
\coordinate (A2) at (-1.25,0,0);
\coordinate (A3) at (0,0,1.25);
\coordinate (A4) at (1.25,0,0);
\coordinate (B1) at (0,1.25,0);
\coordinate (C1) at (0,-1.25,0);

\draw  (A1) -- (A2) -- (B1) -- cycle;
\draw  (A4) -- (A1) -- (B1) -- cycle;
\draw  (A1) -- (A2) -- (C1) -- cycle;
\draw  (A4) -- (A1) -- (C1) -- cycle;
\draw  (A2) -- (A3) -- (B1) -- cycle;
\draw  (A3) -- (A4) -- (B1) -- cycle;
\draw  (A2) -- (A3) -- (C1) -- cycle;
\draw  (A3) -- (A4) -- (C1) -- cycle;

\node at (A1)[place]{};
\node at (A2)[place]{};
\node at (A3)[place]{};
\node at (A4)[place]{};
\node at (B1)[place]{};
\node at (C1)[place]{};

\node [right] at (-0.05,0.1,-1.65) {$1$};
\node [left] at (A2) {$-1$};
\node [left] at (0.05,-0.1,1.65) {$1$};
\node [right] at (A4) {$-1$};
\node [above] at (B1) {$0$};
\node [below] at (C1) {$0$};
\end{tikzpicture}
\]
\caption{$T_4$ with a $D(T_4)$ null-vector labelling}\label{FigT4} 
\end{figure}
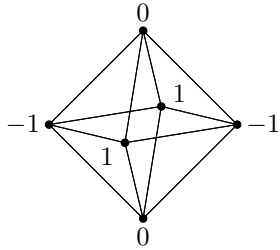

Let $w(a,b)$ be any column of $M(a,b)$. 
We claim that $w(a,b)$ has at least three $a$-entries, and at least three $b$-entries. 
For convenience, we'll refer to vertices corresponding to the $a$-entries of $w(a,b)$ as \emph{$a$-vertices}.
If there are no $a$-vertices, then since $M(a,b) = aX + bY$, 
one of the columns of $X$ is the zero vector. It would follow that the columns of $M(1,0)$ are linearly dependent, 
contradicting Theorem \ref{egnsharp}. 
Thus $w(a,b)$ has at least one $a$ and at least 
one $b$ entry. 

Suppose $w(a,b)$ has at most 2 $a$-entries.
There are three cases we will consider: 
there are two adjacent $a$-vertices, there are two non-adjacent $a$-vertices, or there is
exactly one $a$-vertex. 
In each case, we will construct a null vector $x$ of $D$ which is not orthogonal to $u=w(1,0)$, contradicting
Theorem~\ref{egnsharp}. We will use the labelling in Figure~\ref{FigT4}.

Suppose $w(a,b)$ has two adjacent $a$-vertices. 
Label one of the $a$-vertices with a zero and the adjacent $a$-vertex with $1$. 
Then $x^Tu =1 \neq 0$. 
Suppose $w(a,b)$ has two non-adjacent $a$-vertices. Label the two $a$-vertices
with $1$. Then $x^T u =2 \neq 0$.
Suppose there is only one $a$-vertex in $w$. Label the $a$-vertex with $1$.
Then $x^T u =1 \neq 0$.

Therefore, at least three positions of $w(a,b)$ have the value $a$.
Similarly, at least three positions of $w(a,b)$ must have value $b$.

Since each column of $M(a,b)$ has at least three $a$ and three $b$-entries, there are at least nine $a,b$ pairings
corresponding to each column. Since $M(a,b)$ has three columns, there 
are $27$ $a,b$ column-wise pairs in total. However the number of column-wise $a,b$ pairs in the addressing matrix $M(a,b)$
is simply the number of edges in $\D(T_4)$, namely $18$. This contradiction
implies that $T_4$ is not eigensharp.
\end{proof}

\begin{theorem}\label{tnnotegnsharp}
The triangular graph $T_n$ is not eigensharp for any $n \geq 4$, that is, $N(T_n) \geq n$ for all $n \geq 4$.
\end{theorem}
\begin{proof}
Note that $T_4$ is an induced subgraph of $T_n$ since $K_4$ is an induced subgraph of $K_n$. Let $T$ be an induced subgraph of $T_n$ isomorphic to $T_4$.

Suppose $T_n$ is eigensharp. Let $M(a,b)$ be an eigensharp addressing matrix of $T_n$. By Theorem \ref{egnsharp}, the columns of $M(a,b)$ are orthogonal to any null vector of $D(T_n)$. Let $w$ be one of the columns of $M(a,b)$. We can construct 
a null vector $y$ of $D(T_n)$ by labelling the vertices corresponding to $T$ as described 
in Figure~\ref{FigT4} 
and labelling the remaining vertices of $T_n$ with zeroes. 
In \cite{EGV}, it is described that the columns of an addressing matrix correspond to bicliques that partition
the edgeset of $\D(T_n)$. Every biclique decomposition of $\D(T_n)$ induces a decomposition of $\D(T)$, an induced subgraph of $\D(T_n)$. Lemma \ref{t4notegnsharp} tells us that at least $4$ bicliques are needed to decompose $\D(T)$. Therefore, there must be at least four columns of $M(a,b)$ whose $6$ entries corresponding to $T$ have at least one $a$ and one $b$. The proof of Lemma \ref{t4notegnsharp} guarantees that each of these $4$ vectors, restricted to
the vertices of $T$, has at least three $a$ entries and three $b$ entries. Since $\D(T)$ is an induced subgraph of $\D(T_n)$, there are the same number of edges between the corresponding vertices in the two graphs. However, a contradiction occurs: the eigensharp addressing implies that there are at least 36 edges in $\D(T)$, 
but there are in fact 18. Therefore $T_n$ is not eigensharp.
\end{proof}
 
 For the triangular graph $T_5$ (the complement of the Petersen graph), the following six bicliques partition the edge set of $\mathcal{D}(T_5)$:
\begin{align*}
\{12,13,14,15\} &\cup \{23,24,25,34,35,45\}\\
 \{12,25\} &\cup \{13,14,34,35,45\}\\
 \{23,24\} &\cup \{15,25,34,35,45\}\\
 \{13,23,35\} &\cup \{14,24,45\}\\
 \{15\} &\cup \{12,13,14,34\}\\
 \{34\} &\cup \{25,35,45\}.
\end{align*}
Thus, by Theorem~\ref{tnnotegnsharp}, we know that $5\leq N(T_5)\leq 6$.

\section{Complete multipartite graphs $K_{2,\dots,2}$}\label{K222}

We note here that finding an optimal addressing of  the complete multipartite graph $K_{2,\dots,2}$ with $m$ color classes of size $2$ is a highly non-trivial open problem. It is equivalent to finding the biclique partition number of the multigraph obtained from the complete graph $K_{2m}$ by adding a perfect matching. Motivated by questions in geometry involving nearly-neighborly families of tetrahedra,  this problem  was studied by Zaks \cite{Zaks} and Hoffman \cite{H}. The best current results for $N(K_{2,\dots,2})=bp(\mathcal{D}(K_{2,\dots,2}))$ are due to these authors (the lower bound is due to Hoffman \cite{H} and the upper bound is due to Zaks \cite{Zaks}):
\begin{equation}
m+\lfloor \sqrt{2m}\rfloor-1\leq N(K_{2,\dots,2})\leq \begin{cases}
3m/2-1 & \text{ if } m \text{ is even}\\
(3m-1)/2 & \text{ if } m \text{ is odd}.
\end{cases}
\end{equation}

\section{Open Problems}

We conclude this paper with some open problems.

\begin{enumerate}

\item Must equality hold in (\ref{in:subaddN})
for all choices of $G_i$ ?

\item It is known that determining $bp(G)$ for a graph $G$ is an NP-hard problem (see \cite{KRW}). This problem is NP-hard even when restricted to graphs $G$ with maximum degree $\Delta(G)\leq 3$ (see \cite{C1}). What is the complexity of finding $N(G)$ for general graphs $G$ ? How about graphs with $\Delta(G)> 3$, or other families of graphs ?

\item What is $N(T_n)$ for $n\geq5$ ?

\item The triangular graph $T_n$ is a special case of a Johnson graph. For $n\geq m\geq 2$, the Johnson graph $J(n,m)$ has as its vertex set the $m$-subsets of $[n]$ with two $m$-subsets being adjacent if and only if their intersection has size $m-1$. The Johnson graph is distance-regular and its eigenvalues were determined by Delsarte in his thesis \cite{Delsarte} (see also \cite[Theorem 30.1]{vLW}). Atik and Panigrahi \cite{AP} computed the spectrum of the distance matrix $D(J(n,m))$:
\begin{equation}
    \begin{pmatrix}
    s & 0 & -\frac{s}{n-1}\\
    1 & {n\choose m}-n & n-1
    \end{pmatrix}
\end{equation}
where $s=\sum_{j=1}^{m}j{m\choose j}{n-m\choose j}$. 
Inequality \eqref{Witsenhausen_ineq} implies that $N(J(n,m))\geq n-1$. What is $N(J(n,m))$ ? 

\item The Clebsch graph is the strongly regular graph with parameters $(16,5,0,2)$ that is obtained from the $5$-dimensional cube by identifying antipodal vertices. The eigenvalue bound gives $N\geq 11$ and the connection with the $5$-dimensional cube might be useful to find a good biclique decomposition of the distance multigraph of this graph.

\item What is $N(G)$ if $G$ is a random graph ? Winkler's work \cite{Winkler}, Witsenhausen inequality \ref{Witsenhausen_ineq} and the Wigner semicircle law imply that $n-1\geq N(G)\geq n/2-c\sqrt{n}$ for some positive constant $c$. Recently, Chung and Peng \cite{CP} (see also \cite{A, ABH}) have shown for a random graph $G\in \mathcal{G}_{n,p}$ with $p\leq 1/2$ and $p=\Omega(1)$, almost surely
\begin{equation}
n-o((\log_b n)^{3+\epsilon})\leq \bp(G)\leq n-2\log_b n
\end{equation}
for $b=1/p$ and any positive constant $\epsilon$. Here $\mathcal{G}_{n,p}$ is the Erd\H{o}s-R\'{e}nyi random graph model.

\end{enumerate}

\noindent
\textbf{Acknowledgement.} Some of the initial threads of this project, in particular Section~\ref{CP} and \ref{CPD},
started in conversation with the late D.A. Gregory. We are grateful for his discussion and his leadership over the years.

\end{document}